\documentclass[12pt]{article}
\usepackage{mathtools,amssymb,amsthm}
\usepackage{newtxtext,newtxmath}
\usepackage{bm,pifont}
\usepackage{cases}
\usepackage{color}
\usepackage{graphicx}
\usepackage[tight]{subfigure}
\usepackage{booktabs}
\usepackage{multirow}
\usepackage{makecell}
\usepackage[bookmarks=true,bookmarksnumbered=true,colorlinks=true]{hyperref}
\usepackage[normalem]{ulem}
\usepackage{cancel}
\newtheoremstyle{plain}
  {\topsep}
  {\topsep}
  {\itshape}
  {0pt}
  {\bfseries}
  {~---}
  {5pt plus 1pt minus 1pt}
  {}
\newtheoremstyle{definition}
  {\topsep}
  {\topsep}
  {\normalfont}
  {0pt}
  {\bfseries}
  {~---}
  {5pt plus 1pt minus 1pt}
  {}
\newtheoremstyle{remark}
  {0.5\topsep}
  {0.5\topsep}
  {\normalfont}
  {0pt}
  {\itshape}
  {~---}
  {5pt plus 1pt minus 1pt}
  {}
\newtheoremstyle{note}
  {0.5\topsep}
  {0.5\topsep}
  {\normalfont} % {\itshape}
  {0pt}
  {\bfseries}
  {~---}
  {5pt plus 1pt minus 1pt}
  {}
\theoremstyle{plain}
\newtheorem{theorem}{Theorem}[section]
\newtheorem{proposition}[theorem]{Proposition}

\theoremstyle{definition}
\newtheorem{definition}[theorem]{Definition}
\theoremstyle{remark}

\newtheorem{remark}[theorem]{Remark}
\theoremstyle{note}

\makeatletter
\@addtoreset{equation}{section} 

\makeatother

\DeclareMathOperator*{\argmin}{arg\,min}

\newcommand*{\NN}[1][] {\mathbb{N}^{#1}}

\newcommand*{\RR}[1][] {\mathbb{R}^{#1}}

\newcommand*{\norm}[1]{\left\lVert#1\right\rVert}
\newcommand*{\abs}[1]{\left\vert#1\right\vert}

\makeatletter \newcommand*{\diff}{\@ifnextchar^{\DIfF}{\DIfF^{}}}
\def\DIfF^#1{\mathop{\mathrm{\mathstrut d}}\nolimits^{#1}\gobblespace}
\def\gobblespace{\futurelet\diffarg\opspace}
\def\opspace{\let\DiffSpace\!\ifx\diffarg(\let\DiffSpace\relax%
  \else\ifx\diffarg[\let\DiffSpace\relax%
  \else\ifx\diffarg\{\let\DiffSpace\relax\fi\fi\fi\DiffSpace}
\newcommand*{\deriv}[3][]{\frac{\diff^{#1}#2}{\diff #3^{#1}}}
\newcommand*{\pderiv}[3][]{\frac{\partial^{#1}#2}{\partial #3^{#1}}}

\newcommand*{\xy}[2]{\begin{pmatrix} #1\\ #2 \end{pmatrix}} % (x,y)-column vector

\newcommand*{\ord}[1]{\mathcal{O}(#1)}
\newcommand*{\order}[1][]{\ord{h^{#1}}}

\newcommand*{\braket}[2]{\left<{#1},{#2}\right>}

\newcounter{stepnum}
\newenvironment{step}[1][]{\refstepcounter{stepnum}\par\medskip
   \noindent {\itshape Step~\thestepnum}~--- #1 \rmfamily}{\par\medskip}

\newcommand*{\lac}[1][]{\xi^{#1}}
\newcommand*{\blac}{{\xi^{\alpha}}}
\newcommand*{\csp}{\gamma_{\mathrm{cs}}}
\newcommand*{\ra}{\rho}         % radius of curvature
\newcommand*{\cu}{\kappa}       % signed curvature
\newcommand*{\tu}{\theta}       % turning angle // tangent angle

\newcommand*{\rotrm}[1]{\,\mathrm{R}_{#1}\,}

\begin{document}
\begin{center}
  \Large
  Fairing of planar curves to log-aesthetic curves\\[5mm]

  \normalsize
  Sebasti\'an El\'ias {\sc Graiff Zurita}\\
  Graduate School of Mathematics
  Kyushu  University\\
  744 Motooka, Nishi-ku, Fukuoka 819-0935, Japan\\[2mm] 
  Kenji {\sc Kajiwara}\\
  Institute of Mathematics for Industry, Kyushu University\\
  744 Motooka, Fukuoka 819-0395, Japan\\
  Kenjiro T. {\sc Miura}\\
  Graduate School of Science and Technology, Shizuoka University\\
  3-5-1 Johoku, Hamamatsu, Shizuoka, 432-8561, Japan
\end{center}

\begin{abstract} %
  We present an algorithm to fair a given planar curve by a log-aesthetic curve (LAC). We show how a
  general LAC segment can be uniquely characterized by seven parameters and present a method of
  parametric approximation based on this fact. This work aims to provide tools to be used in reverse
  engineering for computer aided geometric design. Finally, we show an example of usage by applying
  this algorithm to the point data obtained from 3D scanning a car's roof.
\end{abstract}

\section{Introduction} \label{sec:cont.planar} %
In 1994, Harada et al., while investigating the mental images that aesthetically pleasing shapes
suggest, set up an experiment to quantitatively analyze a curve's character from the viewpoint of
the observer \cite{hara94}.  Their main result may be described as follows: the curves that car
designers regard as ``{\em aesthetic}'' have the common property that the frequency histogram of the
radius of curvature follows a piecewise linear relation in a log-log scale.
An analytic formulation of those curves was provided in \cite{miur06,yosh06} defining what will
later be called as the {\em log-aesthetic curves (LAC)}, which promoted theoretical and practical
studies of LAC towards their use in computer aided geometric design as indicated by Levien and
S\'equin \cite{Levien-Sequin:2009}.
In this regard, several works have been written regarding the implementation and construction of LAC
with fixed boundary conditions, see for example \cite{gobi11,gobi14,ziat12}. Furthermore, extensions
to surfaces have also been considered with an emphasis in providing practical tools for industrial
design, see \cite{miur15}.
From a different point of view, LAC have been characterized as curves that are obtained via a
variational principle in the framework of similarity geometry, moreover they can also be seen as
invariant curves under the integrable flow on plane curves which is governed by the Burgers equation
\cite{inog18}.  This fact was also shown to be useful at providing an integrable discretization of
the LAC that preserves the underlying geometric structure \cite{inog21}.
All these previous works contributed to constructing methods that generate a desirable shape with
given fixed conditions. In this paper, in view reverse engineering applications, we provide a method
to characterize a given curve by its closest LAC.  For this purpose, we first show how a general LAC
segment can be uniquely identify by seven parameters, and we then apply a method of parametric
approximation.  Our method is based on the one presented in \cite{Brander-Gravesen-Norbjerg:2017} to
approximate a given planar curve by an Euler's elastica, and in its follow-up work
\cite{GraiffZurita-Kajiwara-Suzuki:2022} for the discrete Euler's elastica.
Finally, we apply our algorithm to the 3D scanner data obtained from an actual car roof and identify
their underlying LAC.

\section{Log-aesthetic curve}
\subsection{Basic formulation}
First of all, let us review some definitions regarding planar curves.  Let $\gamma(s)\in\RR[2]$ be a
planar curve parameterized by arc length $s \in [0, L]$ with total length $L$. We define the tangent
and normal vectors as $T(s) = \gamma'(s)$ and $N(s) = \rotrm{\pi\!/2}T(s)$, respectively, where
$\rotrm{\pi\!/2}$ is a $\pi\!/2$ rotation matrix and ${\text{ }}'=\diff{}/\diff{s}$. By definition,
$\norm{T(s)} = 1$, so the tangent vector can be parameterized as
\begin{equation}
  T(s) = \begin{pmatrix}
    \cos\tu(s)\\
    \sin\tu(s)
  \end{pmatrix},
\end{equation}
where $\tu(s)$ is the turning angle, namely, the angle of the tangent vector measured from the
horizontal axis, so that $\cu(s) = \tu'(s)$ is the curvature and $\ra(s) = 1/\cu(s)$ the radius of
curvature.

As it is shown in \cite{miur06,yosh06}, by considering the analytic formulation of the work by
Harada mentioned in the Introduction, it can be seen that an log-aesthetic curve satisfies
\begin{equation} \label{eq:eq1}
  \log\left(\deriv{s}{R}\right) = \alpha R - \log A,
\end{equation}
for some $\alpha \in \RR$ and $A > 0$, where $R=\log \ra$.  Assuming that functions are well
behaved, we have $\deriv{s}{R} = \ra/\ra'$, then \eqref{eq:eq1} can be rewritten as
\begin{equation} \label{eq:lac}
\ra' \ra^{(\alpha - 1)} = A,
\end{equation}
or equivalently, by taking the derivative of \eqref{eq:lac}, as
\begin{equation} \label{eq:lac.ra}
\ra \ra'' + (\alpha - 1) (\ra')^2 = 0.
\end{equation}
Although \eqref{eq:eq1} is presented in \cite{miur06} as the defining equation of log-aesthetic
curves, we consider \eqref{eq:lac.ra} written in terms of the curvature as an alternative
definition.

\begin{definition}[Log-aesthetic curve]
  \label{def:lac} An arc length parameterized curve $\gamma(s)$ with strictly monotonic radius of
  curvature is called a {\em log-aesthetic curve (LAC)} if its curvature satisfies
  \begin{equation} \label{eq:lac.cu}
    \cu \cu'' - (\alpha + 1) \left(\cu'\right)^2 = 0,
  \end{equation}
  for some constant $\alpha \in \RR$.
\end{definition}

The fundamental theorem of planar curves states that an arc length parameterized planar curve is
uniquely determined by its curvature up to Euclidean transformations.  In addition, the curvature of
an LAC determined by \eqref{eq:lac.cu} has two arbitrary parameters.
In view of this, here we introduce the {\em basic LAC} (Definition \ref{def:blac}) by fixing this
freedom, and show that one can recover arbitrary LAC by applying the similarity transformations and
shift of arc length (Proposition \ref{prop:lac.blac}).

Let us see that for a given LAC, its parameter $\alpha$ is invariant under the similarity
transformations and the reflections.  Noticing that the curvature of a planar curve is invariant
under the Euclidean transformations, we only check the invariance under the scale transformations
and the reflections over the diagonal $\{(x,x) \in \RR[2] | x \in \RR\}$.  For the scale
transformations, consider the arc length parameterized LAC $\gamma(s)$ satisfying \eqref{eq:lac.cu},
for some $\alpha \in \RR$, and define $\tilde\gamma(\tilde{s}) := S \gamma(\tilde{s}/S),$ where
$S > 0$.  Then $\tilde\cu(\tilde{s})$, the curvature of $\tilde\gamma(\tilde{s})$, is given by
$\tilde\cu(\tilde{s}) = S^{-1} \cu(\tilde{s}/S)$ and it satisfies \eqref{eq:lac.cu}.
For the reflections over the diagonal, note that interchanging the $x-$ and $y-$component of the
curve is equivalent to changing the sign of the curvature, and \eqref{eq:lac.cu} is invariant under
that change.

\begin{definition}[Basic LAC] \label{def:blac}
  Let $\blac(s)$ be an LAC defined over an open interval $I \subset \RR$, such that $\{0\} \in I$,
  and satisfying
  \begin{equation} \label{eq:blac.def}
    \begin{cases}
      \cu'(s)= - (\cu(s))^{(\alpha + 1)} < 0, \quad \forall s \in I,\\
      \cu(0) = 1,\\
      \tu(0) = 0,\\
      \blac(0) = 0.
    \end{cases}
  \end{equation}
  We call $\blac(s)$ a {\em basic LAC}.
\end{definition}

Let us see a more explicit expression for the basic LAC and its related quantities.  In what
follows, we use the sub-index $\blac$, as for example $\cu_\blac$, to denote those quantities
associated to their respective basic LAC.  Taking the initial condition into consideration, the
explicit form of the curvature is given by
\begin{equation} \label{eq:blac.cu}
  \cu_{\blac}(s) =
  \begin{cases}
    \exp(-s), & \alpha = 0,\\
    (1 + \alpha s)^{-1/\alpha}, & \alpha \neq 0.
  \end{cases}
\end{equation}
Then, the turning angle is obtained from the curvature by $(\tu_{\blac})' = \cu_{\blac}$,
\begin{equation} \label{eq:blac.tu}
  \tu_{\blac}(s) =
  \begin{cases}
    1 - \exp(-s), & \alpha = 0,\\
    \log(s + 1), & \alpha = 1,\\
    \frac{(1 + \alpha s)^{\frac{\alpha - 1}{\alpha}} - 1}{\alpha - 1}, & \alpha \neq 0, 1.
  \end{cases}
\end{equation}
Finally, the position vector can be obtained from the tangent vector as
\begin{equation} \label{eq:blac.pos}
  \blac(s) = \int_0^s \xy{\cos \tu_{\blac}(\bar{s})}{\sin \tu_{\blac}(\bar{s})} \diff{\bar{s}}.
\end{equation}
Although it is not used in this work, we note that the position vector can be expressed in terms of
the incomplete gamma function, see for example \cite{ziat12}.  For simplicity we consider the case
in which $1 + \alpha s > 0$.  In this case, the maximal interval $I_\alpha \subset \RR$ on which the
basic LAC can be defined is
\begin{equation}
  I_\alpha = \begin{cases}
    \left(-\infty, -1/\alpha\right), & \alpha < 0,\\
    \left(-\infty, \infty\right), & \alpha = 0,\\
    \left(-1/\alpha, \infty\right), & \alpha > 0,
  \end{cases}
\end{equation}
and we assume that all basic LAC are defined over $I_\alpha$.  Finally, note that the image of
$\cu_\blac$ is $\cu_\blac[I_\alpha] = (0, \infty)$.

\begin{proposition} \label{prop:lac.blac} %
  Any LAC with $\alpha \neq 1$ and positive and decreasing curvature can be expressed as a basic LAC
  after applying the similarity transformations and shift of the arc length parameter.  In
  particular, if $\gamma(s)$, $s \in [0, L]$, is an LAC of length $L$, there exists a unique
  $\gamma_0 \in \RR[2]$, $\phi \in [0, 2\pi)$, $S \in \RR\backslash\{0\}$, and $s_0 \in \RR$, such
  that
  \begin{equation} \label{eq:lac.blac}
    \gamma(s) = \gamma_0 + S \rotrm{\phi} \blac(s/S + s_0), \qquad s\in[0, L],
  \end{equation}
  where $\blac(s)$ is a basic LAC of length $l := L/S$.
\end{proposition}

In view of Proposition \ref{prop:lac.blac}, we denote by $\lac[p]$ the LAC segment uniquely
defined by the parameters
\begin{equation}
  p = (\alpha, S, s_0, l, \phi, x_0, y_0),
\end{equation}
where $\gamma_0 = (x_0, y_0)$.  With this notation, note that $\lac[p] = \blac$ in the case that
$p = (\alpha, 1, s_0, l, 0, 0, 0)$ for some $s_0$ and $l$ satisfying that
$(s_0, s_0 +l) \subset I_\alpha$.

\begin{proof}[Proof of Proposition \ref{prop:lac.blac}] %
  For a given LAC $\gamma(s)$, $s \in [0, L]$ with $\alpha \neq 1$ and positive and decreasing
  curvature, from Definition \ref{def:lac} its curvature satisfies \eqref{eq:lac.cu}, which can be
  integrated once to obtain
  \begin{equation} \label{eq:prop.cur}
    \cu'(s) = - A (\cu(s))^{(\alpha + 1)},
  \end{equation}
  for some $A > 0$.  Next, we consider the curve $\bar{\gamma}(\bar{s}) := S^{-1}\gamma(\bar{s}S)$,
  $\bar{s} \in [0, L/S]$, and we set $S = A^{1/(\alpha-1)}$.  We see that its curvature satisfies
  \begin{equation} \label{eq:prop.lac.blac}
    \begin{cases}
      \bar\cu'(\bar{s})= -(\bar\cu(\bar{s}))^{(\alpha + 1)},\\
      \bar\cu(0) = A^{1/(\alpha-1)} \cu(0),
    \end{cases}
  \end{equation}
  which can be integrated to obtain
  \begin{equation}
    \bar\cu(\bar s) =
    \begin{cases}
      \exp(-(\bar s - \log \bar\cu(0))), & \alpha = 0,\\
      \left[1 + \alpha \left(\bar s + \frac{(\bar\cu(0))^{-\alpha} -
            1}{\alpha}\right)\right]^{-1/\alpha}, & \alpha \neq 0.
    \end{cases}
  \end{equation}
  By comparing $\bar\cu$ with $\cu_\blac$ \eqref{eq:blac.cu} there exists a unique $s_0 \in \RR$
  such that $\bar\cu(\bar{s}) = \cu_{\blac}(\bar{s} + s_0)$, and from the fundamental theorem of
  planar curves it follows the curves $\bar\gamma$ and $\blac$ are congruent up to rigid
  transformations, i.e.
  \begin{equation}
    \bar\gamma(\bar s) = \bar\gamma_0 + \rotrm{\phi} \blac(\bar s + s_0), \qquad \bar s\in[0,
    L/S],
  \end{equation}
  for some $\bar\gamma_0\in\RR[2]$ and $\phi\in[0, 2\pi)$. Finally, we use that
  $\gamma(s) = S \bar\gamma(s/S)$ to obtain \eqref{eq:lac.blac} with $\gamma_0 := S \bar\gamma_0$.
\end{proof}

\begin{remark} \label{rem:alpha1} %
  In the proof of Proposition \ref{prop:lac.blac}, the scale
  transformation is used to alter the value of $A$ in \eqref{eq:prop.cur} without changing the value
  of $\alpha$.  However, in the case $\alpha = 1$ this technique cannot be exploited.  This is due
  to the fact that $\alpha = 1$ corresponds to the logarithmic spiral, which is a self-similar
  curve.  Particularly, $A$ in \eqref{eq:prop.cur} is invariant under the scale transformations,
  thus the technique used in Proposition \ref{prop:lac.blac} cannot be used to recover the entire
  family of LAC with $\alpha = 1$.
\end{remark}

\begin{remark}
  Let $X$ be the reflection of $\RR[2]$ defined by the map $(x,y) \mapsto (y,x)$.  If $\gamma(s)$
  with $s \in [0, L]$ is an LAC, then also are
  \begin{equation} \label{eq:lac.transf}
    \begin{cases}
      \gamma^{(1)}(s) := \gamma(L - s),\\
      \gamma^{(2})(s) := X\gamma(s),\\
      \gamma^{(3)}(s) := X\gamma(L -s).
    \end{cases}
  \end{equation}
  Moreover, their respective curvatures satisfy
  \begin{equation}
    \begin{cases}
      \cu^{(1)}(s) = -\cu(L - s),\\
      \cu^{(2)}(s) = -\cu(s),\\
      \cu^{(3)}(s) = \cu(L -s),
    \end{cases}
  \end{equation}
  which allow us to use Proposition~\ref{prop:lac.blac} in those cases in which the curvature is not
  positive and decreasing, by applying one of the transformations \eqref{eq:lac.transf}.
\end{remark}

\subsection{Recovering the parameters of an LAC segment} \label{sec:lac.guess}
We focus our attention to the problem of finding the parameters that uniquely identify a given LAC
segment.  We proceed in three steps, where we solve several linear equations in the least squares
sense, with the objective of constructing an algorithm that can be applied to general curves.
Given an LAC segment $\gamma(s)$, $s \in [0, L]$, by possibly applying one of the transformations
\eqref{eq:lac.transf}, we assume that its curvature is positive and strictly monotonic decreasing.
From Proposition \ref{prop:lac.blac}, there exists $p = (\alpha, S, s_0, l, \phi, x_0, y_0)$ such
that $\gamma(s) = \lac[p](s)$, $s \in [0, L]$.  %

\begin{remark}
  In view of Remark \ref{rem:alpha1} we omit the case $\alpha = 1$, and for simplicity in the
  formulation of this method we further omit the case $\alpha = 0$.  Removing these values does not
  hinder the quality of the algorithm, because they are only single points on the real line.
\end{remark}

We start by defining and recalling some useful quantities.  From \eqref{eq:lac.blac},
\begin{equation}
  \gamma(s) = \xy{x_0}{y_0} + S \rotrm{\phi} \blac(s/S + s_0), \qquad s\in[0, L],
\end{equation}
implies that
\begin{equation} \label{eq:cu.bcu}
  \cu(s) = S^{-1} \cu_{\blac}(s/S + s_0).
\end{equation}
Recall that  $R = - \log \cu$, hence
$R(s) = \log S + R_{\blac}(s/S + s_0)$ and using \eqref{eq:blac.def} we have that
$\log (R_{\blac}') + \alpha R_{\blac} = 0$, thus we obtain
\begin{equation} \label{eq:alg1}
  \log (R') + \alpha R = (\alpha - 1) \log S.
\end{equation}

\begin{step} \label{step1}
  Let $c_1 := \alpha$ and $c_0:= (\alpha - 1) \log S$.  Then, from \eqref{eq:alg1}, in the least
  squares sense we have
  \begin{equation}
    (c_0, c_1) = \argmin_{(\bar c_0,\bar c_1)}\left\{\frac{1}{2} \int_0^L (\log (R') + \bar c_1 R -
      \bar c_0)^2 \diff{s} \right\},
  \end{equation}
  which leads to
  \begin{equation} \label{eq:alg.c_1}
    c_1 = \frac{L \int_0^L R \log (R') \diff{s} - \int_0^L R \diff{s} \int_0^L \log (R') \diff{s}}%
    {(\int_0^L R \diff{s})^2 - L \int_0^L R^2 \diff{s}},
  \end{equation}
  and
  \begin{equation} \label{eq:alg.S}
    c_0 = \frac{1}{L} \int_0^L (\log (R')  + c_1 R) \diff{s}.
  \end{equation}
  Then, $\alpha = c_1$ and $S = \exp(c_0/(c_1 - 1))$.
\end{step}

\begin{step} \label{step2}
  From \eqref{eq:blac.cu} and \eqref{eq:cu.bcu}, we have that
  $\cu(s) = S^{-1} (1 + \alpha (s/S + s_0))^{-1/\alpha},$
  which allows us to isolate the parameter $s_0$ as
  \begin{equation}
    s_0 = \frac{(S \cu(s))^{-\alpha}}{\alpha} - \frac{1}{\alpha} - \frac{s}{S}.
  \end{equation}
  Then,
  \begin{equation}
    s_0 := \argmin_{\bar s_0}\left\{ \frac{1}{2} \int_0^L \left(\frac{1}{\alpha S^\alpha
          \cu(s)^\alpha} - \frac{1}{\alpha} - \frac{s}{S} - \bar s_0\right)^2 \diff{s}\right\}
  \end{equation}
  gives
  \begin{equation} \label{eq:alg.s0}
    s_0 = \frac{1}{\alpha L S^\alpha} \int_0^L (\cu(s))^{-\alpha} \diff{s} - \frac{1}{\alpha}
    - \frac{1}{LS} \int_0^L s \diff{s}.
  \end{equation}
  Similarly, we compute $s_{\mathrm{end}} := L/S + s_0$ from
  $\cu^{(3)}(s) = \cu(L-s) = S^{-1} (1 + \alpha (s_{\mathrm{end}} - s/S))^{-1/\alpha}$.  In the
  least squares sense, we obtain
  \begin{equation} \label{eq:alg.s1}
    s_{\mathrm{end}} = \frac{1}{\alpha L S^\alpha} \int_0^L (\cu(L-s))^{-\alpha} \diff{s}
    - \frac{1}{\alpha} + \frac{1}{LS} \int_0^L s \diff{s},
  \end{equation}
  Hence,
  % equivalently,
  \begin{equation}
    l = s_{\mathrm{end}} - s_0 = \frac{2}{L S} \int_0^L s \diff{s}.
  \end{equation}
\end{step}

\begin{step} \label{step3}
  At this point, we are left with finding the rotation and translation parameters.  For the former,
  note that the angle function of $\gamma$ and $\blac$ differ only by a constant $\phi$,
  as
  \begin{equation}
    \tu(s) = \phi + \tu_{\blac}(s/S + s_0).
  \end{equation}
  Thus, in the least squares sense we obtain
  \begin{equation} \label{eq:alg.phi}
    \phi = \frac{1}{L} \int_0^L(\tu(s) - \tu_{\blac}(s/S + s_0))\diff{s}.
  \end{equation}
  Finally, for the translation $(x_0, y_0)$ we solve \eqref{eq:lac.blac} in the least squares sense,
  \begin{equation} \label{eq:alg.gamma0}
    \xy{x_0}{y_0} = \frac{1}{L} \int_0^L \left(\gamma(s) - S \rotrm{\phi} \blac(s/S + s_0)\right) \diff{s}.
  \end{equation}
\end{step}

\section{Approximation of planar curves}
\subsection{Methodology}
Now we consider the case where a general curve segment is given and we want to find an LAC segment
that is the closest in a $L^2$-distance sense.  For the applications that we have in mind, the input
is always a discrete curve and thus we must consider a discretization of the LAC.
From previous section, we know that any LAC segment, after a possible change of parameterization or
reflection, can be expressed as
\begin{equation} \label{eq:lact1}
  \lac[p](s) = \xy{x_0}{y_0} + S \rotrm{\phi} \blac(s/S + s_0),\qquad s\in[0, L].
\end{equation}
Using the equations for the basic LAC \eqref{eq:blac.cu}, \eqref{eq:blac.tu} and
\eqref{eq:blac.pos}, we write
\begin{equation}
  \lac[p](s) = \xy{x_0}{y_0} + \int_0^s \xy{\cos(\tu_\blac(t/S+s_0) + \phi)}{\sin(\tu_\blac(t/S+s_0)
+ \phi)} \diff{t},\qquad s\in(0,L).
\end{equation}
Let $N \in \NN$, and define
\begin{equation} \label{eq:lac.approx.param}
  \begin{cases}
    h := \frac{L}{N-1},\\
    z := \frac{L}{S(N-1)},
  \end{cases}
\end{equation}
and the discrete curve $\lac[\Theta]_n \in \RR[2]$, $n = 0, \dots, N-1$, such that
\begin{equation} \label{eq:reconstruct_lac}
  \left\{
    \begin{array}{l}
      \lac[\Theta]_n = %
      \lac[\Theta]_{n-1} + h \xy{\cos(\tu_\blac(z n+s_0) + \phi)}{\sin(\tu_\blac(z n+s_0) + \phi)},%
      \qquad n=1,\dots,N-1,\\
      \lac[\Theta]_0 = \xy{x_0}{y_0},
    \end{array}
  \right.
\end{equation}
where we use introduce the notation $\lac[\Theta]$ for the discrete curve that depends on the
parameters
\begin{equation} \label{eq:param.lac_n}
  \Theta := \left(x_0, y_0, h, \phi, z, s_0, \alpha\right).
\end{equation}
Note that, the recursive expression \eqref{eq:reconstruct_lac} gives an approximation of $\lac[p](s)$
of second order, in the sense that it satisfies
\begin{equation}
  \lac[p](hn) = \lac[\Theta]_n + \order[2], \qquad n = 0, \dots, N-1.
\end{equation}
The input curve is a list of $N$ equidistant $(x,y)$-points, which we express as the discrete curve
\begin{equation}
  \gamma_n = \xy{x_n}{y_n},\qquad n=0,\dots, N-1.
\end{equation}
Given $\gamma_n$ we look for the closest $\lac[\Theta]_n$ in the $L^2$-distance sense, i.e., we seek
to find a set of parameters $\Theta^*$ such that
\begin{equation} \label{eq:opt}
  \Theta^* = \argmin_{\Theta \in U} \left\{\frac{1}{2}
    \sum_{n=0}^{N-1} \norm{\lac[\Theta]_n - \gamma_n}^2 \right\},
\end{equation}
with the admissible set $U$ given by
\begin{equation} \label{eq:lac.U}
  U = \left\{(x_0, y_0, \phi, h, z, s_0, \alpha) \in \RR[7] \,\bigg| %
    \begin{array}{l}
      \scalebox{0.9}{$(x_0, y_0)\in \RR[2], h>0, \phi \in [0, 2\pi)$,}\\
      \scalebox{0.9}{$z > 0,  s_0 \in I_\alpha\text{ and }\alpha \in \RR\backslash\{0, 1\}$}
    \end{array}\right\}.
\end{equation}
The optimization problem is solved using the Interior Point Optimizer (IPOPT) package, which for our
purpose can be seen as a gradient descent-like method for nonlinear optimizations (see \cite{ipopt}),
so we need to compute the gradient of the objective function in \eqref{eq:opt},
\begin{equation} \label{eq:opt.obj}
  \mathcal{L}(\Theta) := \frac{1}{2} \sum_{n=0}^{N-1} \norm{\lac[\Theta]_n - \gamma_n}^2.
\end{equation}
We have,
\begin{equation}
  \pderiv{}{\Theta_i}\mathcal{L}(\Theta) = \sum_{n=0}^{N-1} \braket{\lac[\Theta]_n - \gamma_n}{\pderiv{}{\Theta_i}\lac[\Theta]_n}, %
  \qquad \Theta_i = x_0, y_0, \phi, h, z, s_0, \alpha,
\end{equation}
which is computed recursively from equation \eqref{eq:reconstruct_lac}, using that
\begin{equation}
  \pderiv{}{\Theta_i}\lac[\Theta]_n = \pderiv{}{\Theta_i}\lac[\Theta]_{n-1} +
  \begin{cases}
    0 & \Theta_i = x_0, y_0,\\
    T_n & \Theta_i = h,\\
    h \pderiv{}{\Theta_i}T_n & \text{otherwise},
  \end{cases}
\end{equation}
with
\begin{equation}
  \pderiv{}{\Theta_i}\lac[\Theta]_0 =
  \begin{cases}
    \left(\begin{smallmatrix}1\\0\end{smallmatrix}\right) & \Theta_i = x_0,\\
    \left(\begin{smallmatrix}0\\1\end{smallmatrix}\right) & \Theta_i = y_0,\\
    \left(\begin{smallmatrix}0\\0\end{smallmatrix}\right) & \text{otherwise},
  \end{cases}
\end{equation}
where we denoted $T_n$ as
\begin{equation}
  T_n := \xy{\cos(\tu_\blac(z n+s_0) + \phi)}{\sin(\tu_\blac(z n+s_0) + \phi)}.
\end{equation}
Then, using that $(\tu_\blac)' = \cu_\blac$, we obtain the gradient of $T_n$ as
\begin{equation}
  \pderiv{}{\Theta_i} T_n = \rotrm{\pi/2} T_n \times%
  \begin{cases}
    1, & \Theta_i = \phi,\\
    \cu_\blac(z n + s_0), & \Theta_i = s_0,\\
    n\cu_\blac(z n + s_0), & \Theta_i = z,\\
    \pderiv{}{\alpha}\tu_\blac(z n + s_0), & \Theta_i = \alpha,
  \end{cases}
\end{equation}
where
\begin{equation}%
  \pderiv{}{\alpha}\theta_\blac(s) = \frac{(1+s\alpha)^{-\frac{1}{\alpha}}\Big(%
    (\alpha-1)(1+s\alpha)\log(1+s\alpha) -
    \alpha\big(\alpha+2s\alpha-s\big)\Big)}{\alpha^2(\alpha-1)^2} + %
  \frac{1}{(\alpha-1)^2},%
\end{equation}
which is obtained by direct computation from \eqref{eq:blac.tu}.  We compute the initial guess
$\bar\Theta$ by using a discrete analogue of the three steps described in
Section~\ref{sec:lac.guess}.  We proceed as follows; we first approximate the curvature at each
point $\gamma_n$, $n = 1, \dots, N-2$, by
\begin{equation} \label{eq:disc.cu}
  \cu_n = \frac{2}{h} \frac{\det(T_{n-1}, T_n)}{1 + \braket{T_{n-1}}{T_n}},\qquad n = 1, \dots, N-2,
\end{equation}
which is obtained from the inscribe circle tangent at the middle point of two consecutive segments
\cite{Hoffmann:MI_LN}, where $T_n = (\gamma_{n+1} - \gamma_n)/h$.  Then, let us define the logarithm
of the radius of curvature by
\begin{equation}
  R_n = - \log \cu_n,\qquad n = 1, \dots, N-2,
\end{equation}
and its discrete derivative by
\begin{equation}
  \Delta R_n = -\frac{\cu_{n+1} - \cu_n}{\cu_n}, \qquad n = 1, \dots, N-3.
\end{equation}
From Step \ref{step1}, we obtain
\begin{equation} \label{eq:alg.alpha}
  \bar \alpha = %
  \frac{(N-3)\sum_{n=1}^{N-3}R_n\log\Delta R_n-\sum_{n=1}^{N-3}R_n\sum_{n=1}^{N-3}\log\Delta R_n}%
  {\left(\sum_{n = 1}^{N-3}R_n\right)^2 - (N-3) \sum_{n = 1}^{N-3}R_n^2},
\end{equation}
and
\begin{equation}
  \bar S = h^{\frac{1}{1 - \bar\alpha}}\exp\left(\frac{1}{(\bar\alpha - 1)(N-3)}%
    \sum_{n = 1}^{N-3}\left(\log\Delta R_n  + \bar\alpha R_n\right)\right).
\end{equation}
From Step \ref{step2}, and using that $\bar{z} = \bar{l}/(N-1)$, we obtain
\begin{equation}
  \bar s_0 = \frac{1}{\bar\alpha (N-2) \bar{S}^{\bar\alpha}} \sum_{n=1}^{N-2}(\cu_n)^{-\bar\alpha} -
  \frac{1}{\bar\alpha}  - \frac{(N-1) h}{2 \bar{S}},
\end{equation}
and
\begin{equation}
  \bar{z} =  \frac{(N-1) h}{\bar{S}},
\end{equation}
Finally, from Step \ref{step3}, we obtain
\begin{equation}
  \bar\phi = \frac{1}{N-1} \sum_{n=0}^{N-2} (\tu_n - \tu_{\lac[\bar\alpha]}(\bar{z}n + \bar{s}_0)),
\end{equation}
and
\begin{equation}
  \xy{\bar{x}_0}{\bar{y}_0} = %
  \frac{1}{N} \sum_{n=0}^{N-1} \left(\gamma_n - \lac[\check\Theta]_n\right).
\end{equation}

\begin{figure}[ht]
  \centering
  \includegraphics[width=8cm]{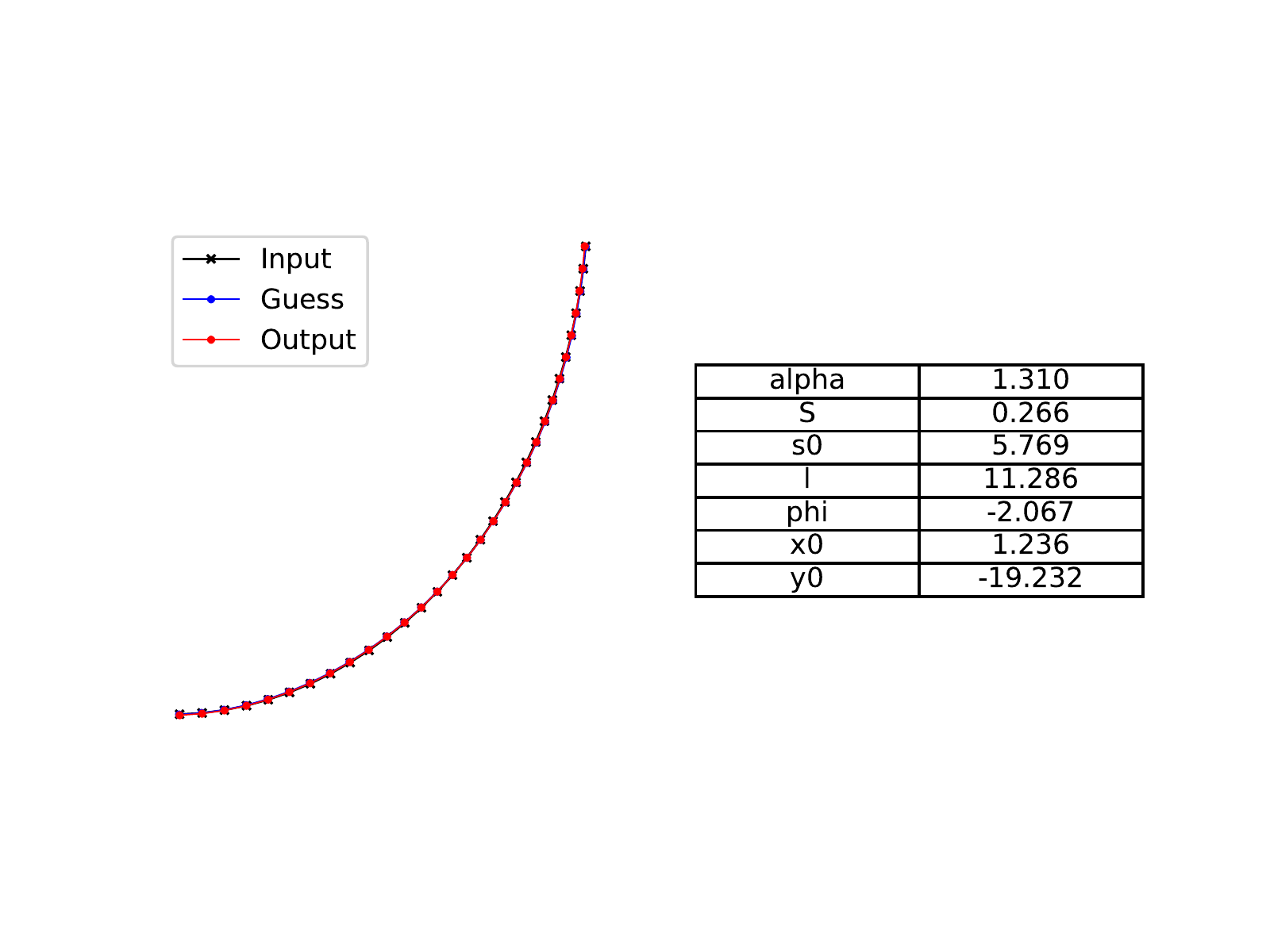}
  \includegraphics[width=8cm]{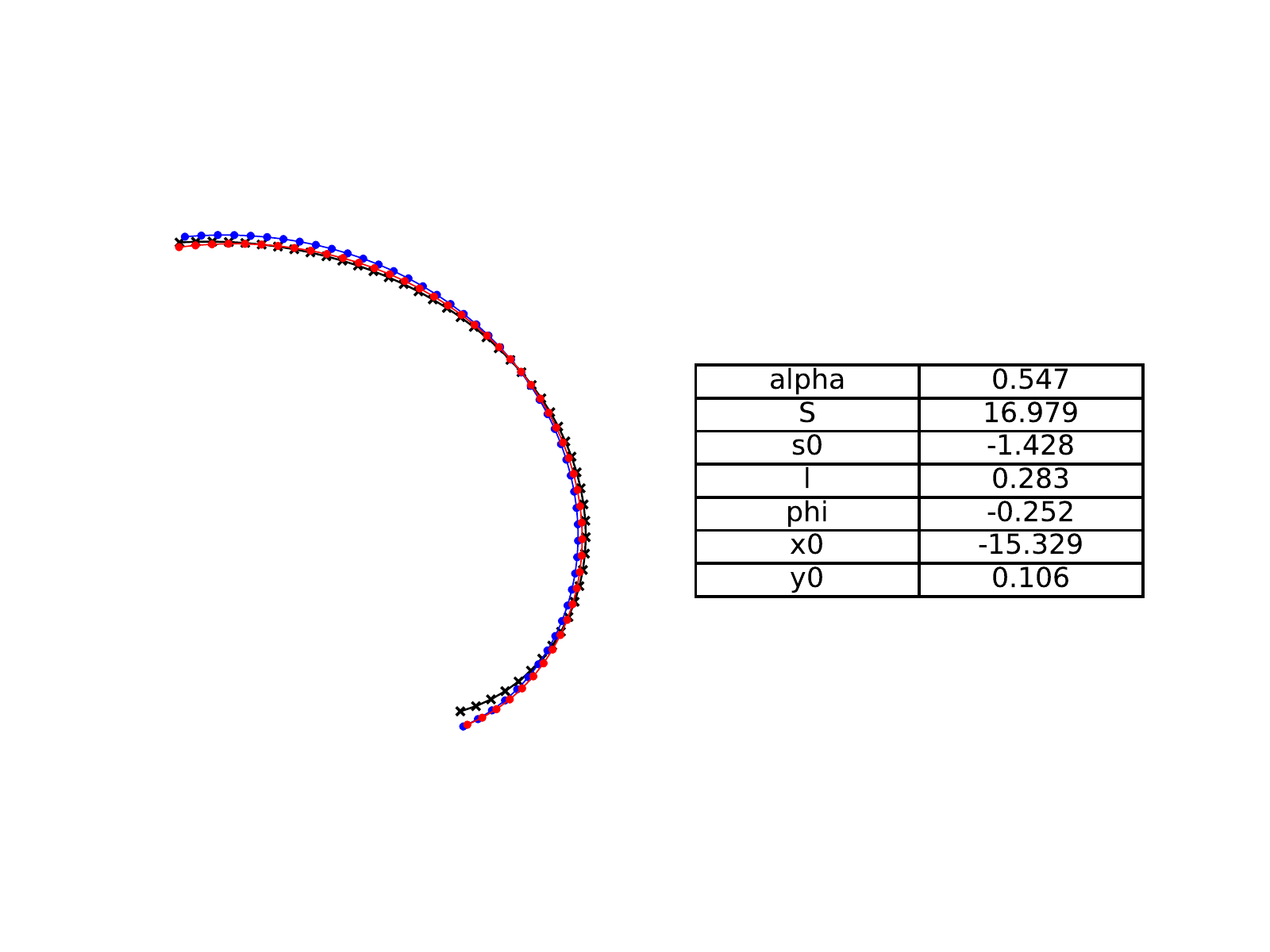}
  \includegraphics[width=8cm]{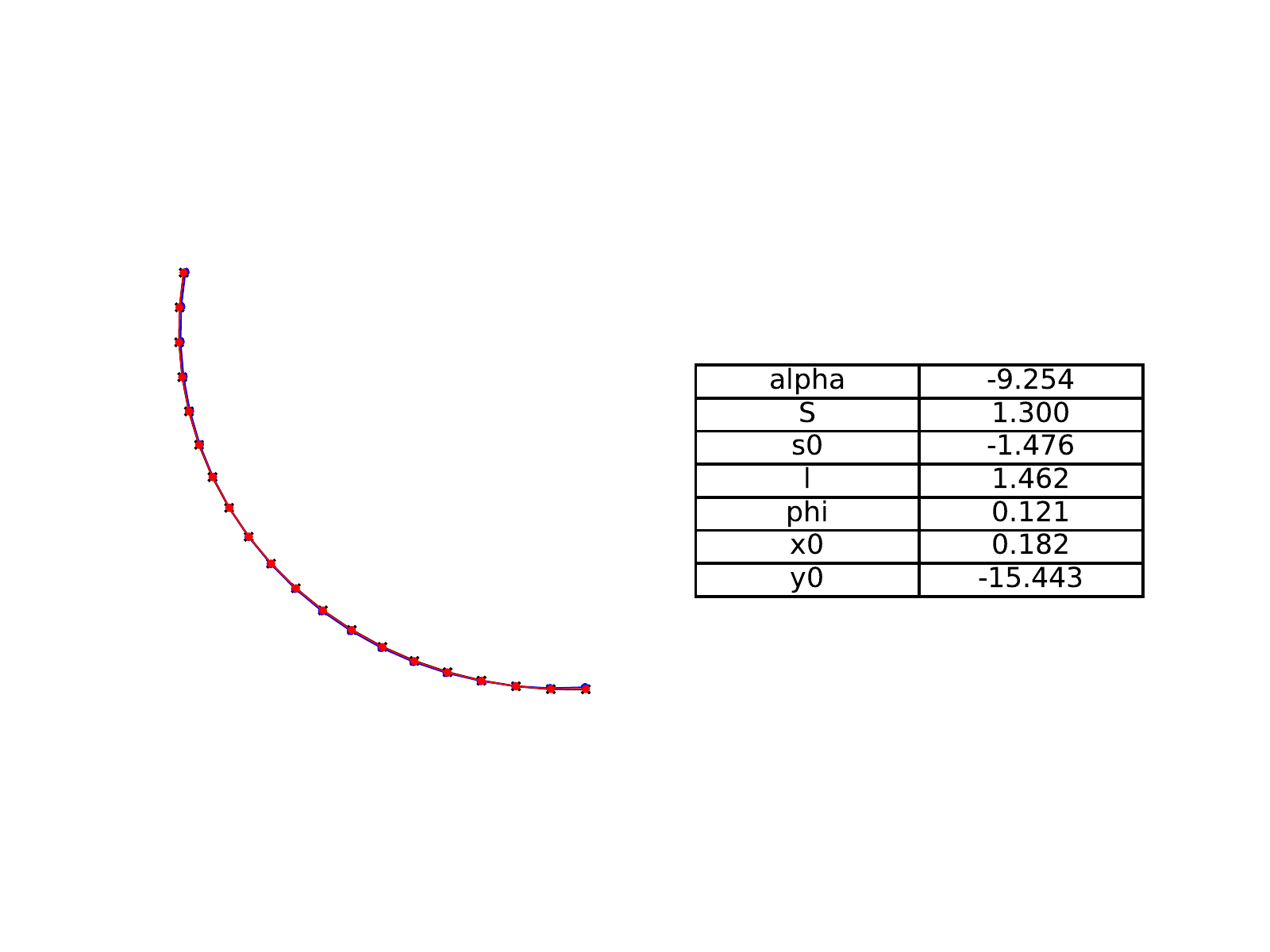}
  \includegraphics[width=8cm]{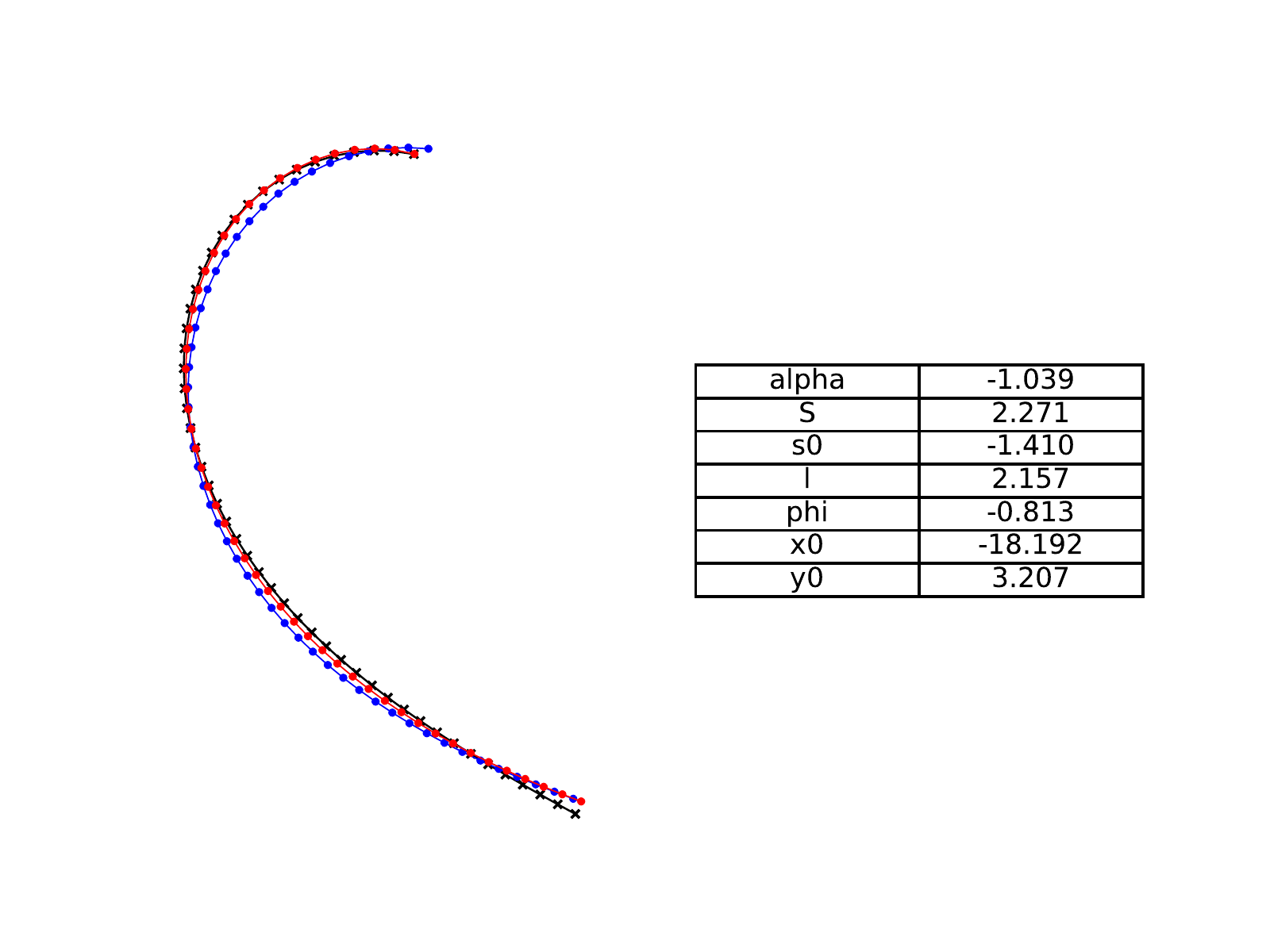}
  \caption{Examples of discrete curves approximated by LAC. Black curves are the input curves,
      blue curves are the first guesses obtained from our algorithm and the red curves are the final
      outputs of the IPOPT program.}
  \label{fig:approx.examples}  
\end{figure}

As a first test of this algorithm we use synthetic data: discrete curves with constant step size
based on B\'ezier curves.  These curves were split in segments with sign preserving monotonic
curvature.  Then, we found the initial guess parameter, plot the resulting curve, and finally
applied the $L^2$-distance minimization algorithm. Some examples are shown in Figure
\ref{fig:approx.examples}.  In the next section we apply this method to real data.

\subsection{Application}
In order to test the fairing algorithm shown in the previous section, we characterize some simple
profile lines of a car's roof (Toyota Prius).  A 3D model was obtained by measuring a scale model
car with a 3D laser scanner (Hexagon 8330-7), see Figure \ref{fig:measuremnt}.  This 3D model, was
stored in STL (Standard Triangle/Tessellation Language) file, which encodes the geometry of the
object in a triangular mesh.  Using Rhinoceros 6, a computer-aided design software, we intercepted
the 3D model with vertical planes, see Figure~\ref{fig:sampledcurves}.  Finally, we projected the
curves into the plane and processed the discrete point to obtain a planar discrete curve with
constant step size.  The length of each curve is approximately $1500$ mm, and the separation between
curve is $100$ mm.
\begin{figure}[ht]
  \centering
  \begin{tabular}{lcr}
    \includegraphics[height=6cm]{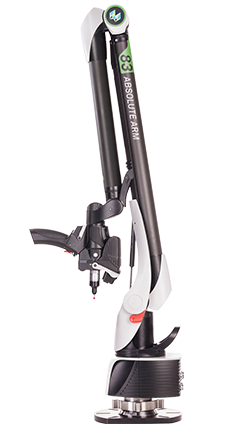} &&
    \raisebox{0.6\height}{\includegraphics[width=6cm]{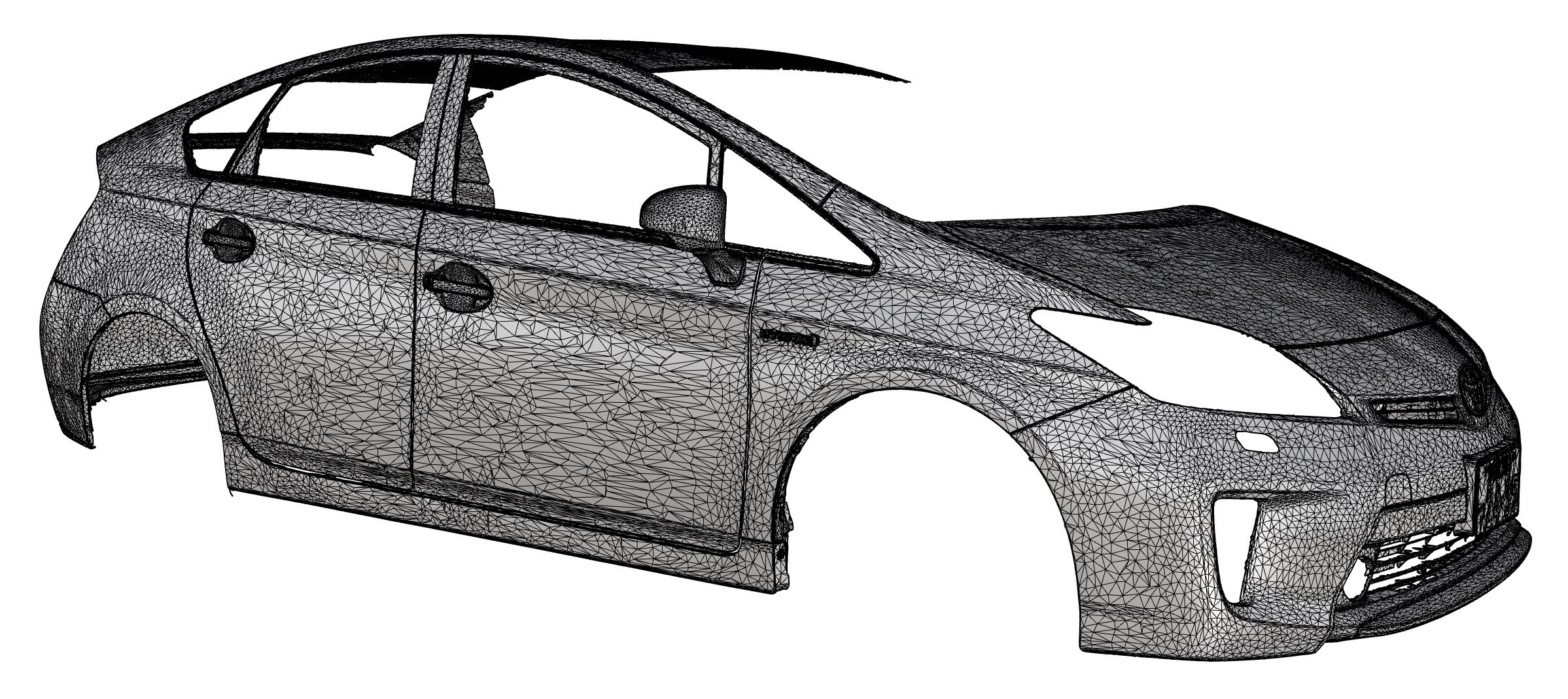}}
  \end{tabular}
  \caption{Left: Hexagon 8330-7 (7 axis arm type 3D laser scanner with measurement accuracy of
      $0.078$ mm). Right: 3D model (STL data).}
  \label{fig:measuremnt}
\end{figure}
\begin{figure}[ht]
  \centering
  \includegraphics[height=5cm]{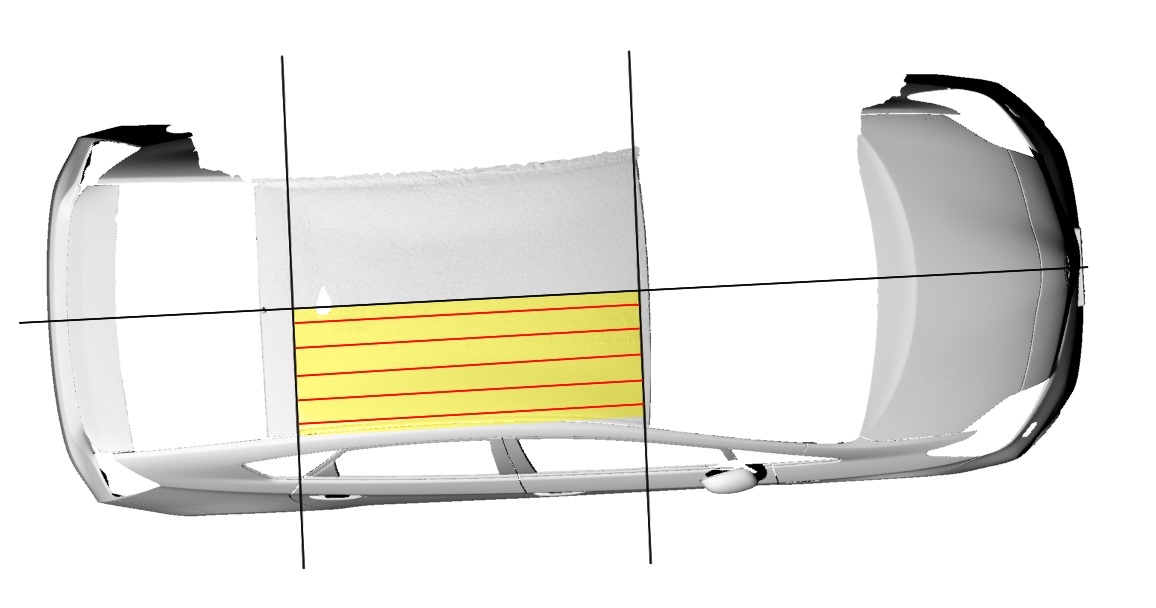}
  \caption{Position of sampled curves, obtained from the interception of the 3D model and
      vertical planes (Software: Rhino 6).  The black lines represent the intercepting planes, and
      the red lines, in the highlighted area, represent the resulting curves. From the center to the
      bottom, curves are labelled as \textsf{ps\_1}, \textsf{ps\_2}, \textsf{ps\_3}, \textsf{ps\_4}
      and \textsf{ps\_5}.}
  \label{fig:sampledcurves}
\end{figure}
We observed that the curvature of the keylines is highly irregular, as a product of the measuring
technique employed.  To reduce the noise, we proceeded as follows.  Let the curve
$\gamma_n \in \RR[2]$, $n = 0, \dots, N-1$ with a constant step size
$\norm{\gamma_{n+1} - \gamma_n} = h > 0$, be the raw data; then:
\par\medskip\noindent
{\bf (0)}~---~%
Let $\check N = 3$.
\par\medskip\noindent
{\bf (1)}~---~%
For $\check N < N$, apply the Ramer--Douglas--Peucker algorithm (see \cite{doug73,rame72}) to
$\gamma_n$, to obtain a new curve $\check \gamma_n$, $n = 0, \dots, \check N-1$, such that
$\check\gamma_0 = \gamma$ and $\check\gamma_{\check N -1} = \gamma_{N-1}$.
\par\medskip\noindent
{\bf (2)}~---~%
Construct a cubic spline curve $\csp(t)$, $t \in [0, L]$ using $\check\gamma_n$,
$n = 0, \dots, \check N-1$, as the control points; hence, $\csp(0) = \gamma_0$ and
$\csp(L) = \gamma_{N-1}$.
\par\medskip\noindent
{\bf (3)}~---~%
Construct a discrete curve $\bar\gamma_n$, $n = 0, \dots, N-1$ with step size $h$, by sampling the
cubic spline $\csp(t_n)$ in a appropriate manner such that $\norm{\csp(t_{n+1}) - \csp(t_n)} = h$
and $\bar\gamma_n = \csp(t_n)$.
\par\medskip\noindent
{\bf (4)}~---~%
If the residual $R$ is greater than a prescribed error $\eta$,
\begin{equation}
  R = \frac{1}{L^2} \sum_{n=0}^{N-1}\norm{\bar\gamma_n - \gamma_n}h > \eta,
\end{equation}
repeat from step (1), with a greater value of $\check N$. Otherwise, the process ends.

\begin{figure}[ht]
  \centering
  \includegraphics[width=13cm]{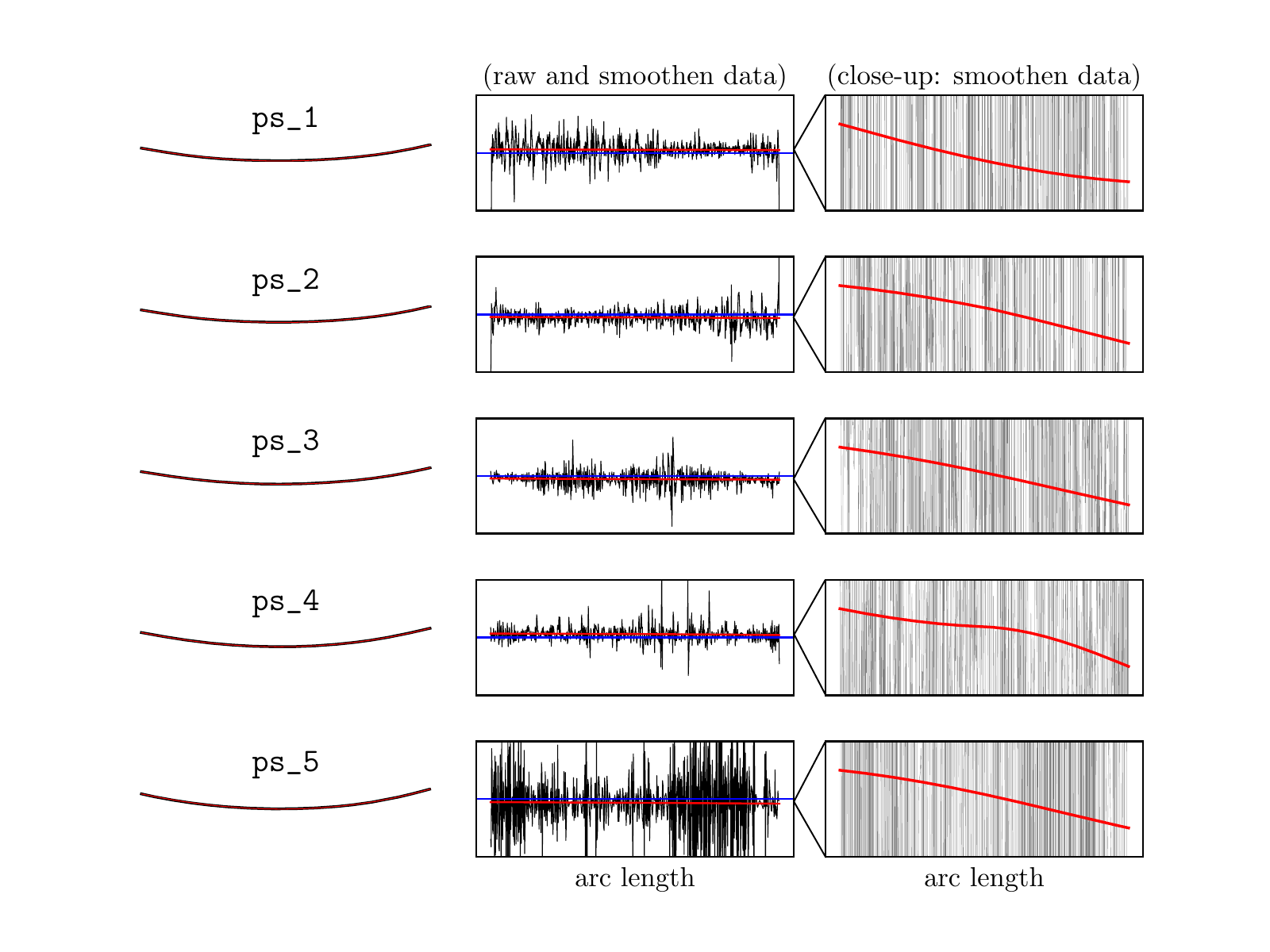} %width=18cm
  \caption{Input curves without smoothing. Left column: Each curve represent a different
      section of the car's roof. Black lines are raw data and red lines are the smoothen data (the
      input curves after noise reduction).  Left column: raw and smoothen curves in the plane,
      practically overlapping with each other.  Middle column: Curvature plot vs arc length for
      their corresponding input curves, where the blue line is a reference for the constant
      $\cu = 0$.  Right column: same as middle column, with a close-up on the smoothen data.}
  \label{fig:app.input}  
\end{figure}

\par\medskip
In our case, we used $\eta = 10^{-3}$, and we observed that this produces a smoother plot for the
curvature, see Figure \ref{fig:app.input}.  Hence, we use $\bar\gamma_n$, instead of $\gamma_n$, as
the input curve for the approximation algorithm.  However, we noticed that the admissible set $U$
for the parameters $\Theta$, as defined in \eqref{eq:lac.U}, was too broad and unexpected jumps in
the value of the parameters produced unrealistic outputs.  To keep the optimization relatively close
to the initial guess, we decided to constrain the admissible set to $U[\bar\Theta]$, defined by
\begin{equation}
  \bar U_{0.1} =%
  U \cap \left\{\Theta\in\RR[7] : \abs{\Theta_i - \bar\Theta_i} < 0.1\bar\Theta_i\right\},
\end{equation}
where $\bar\Theta$ is the initial guess for the IPOPT method.  In this way, $\bar U_{0.1}$ constrain
the final result to be in a $10\%$ range of the initial guess.  Final results are presented in
Figure~\ref{fig:app.result}.  We note that the parameters that we obtained provide a good fit of the
input curve.  We can observe that, despite the curves being similar in shape to each other, the
values of the parameter $\alpha$ have big variations.  However, the final result is still a good
fit, after finding the remaining parameters.  It is interesting to note that, the algorithm that we
provided allow us to input the parameter $\alpha$ by hand (or by replacing \eqref{eq:alg.alpha} by
an alternative expression or algorithm) and then we can continue with the next steps without any
further change.
We conclude that the method proposed has a good performance, however further analysis on the
recovery of the parameter $\alpha$ is required.
% For the applications that we have in mind, it is important to discern between different values of
% $\alpha$, and being able to characterize a given curve by this parameter.

\begin{figure}[ht]
  \centering
  \includegraphics[width=13cm]{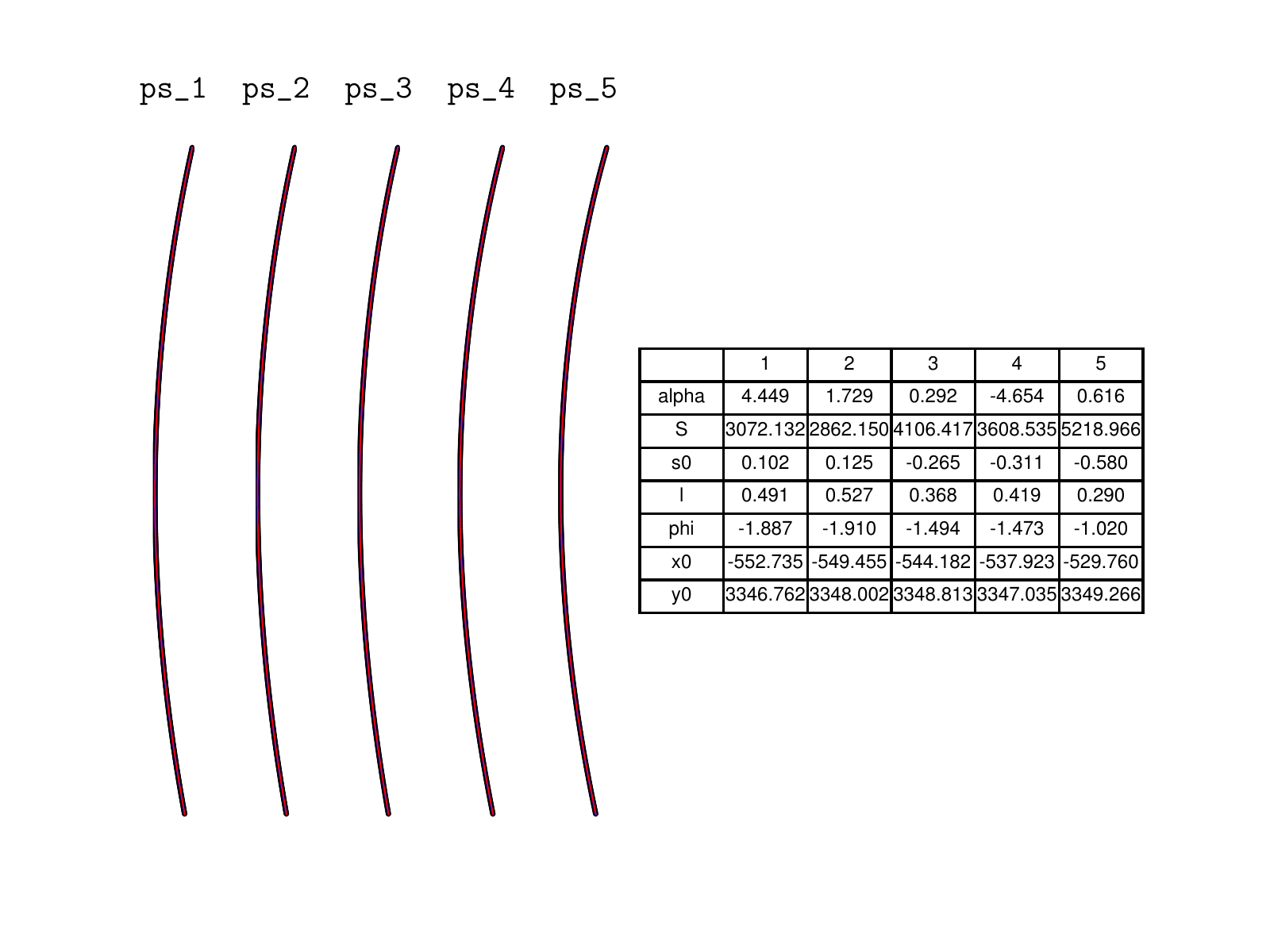}
  \caption{Discrete curves approximated by LAC. Each curve represent a different section of
      half of the car roof, taken at $100$ mm apart. Black lines: input curves. Red lines: LAC
      output.}
  \label{fig:app.result}  
\end{figure}

\section{Concluding remarks}
In the present work we briefly reviewed the notion of log-aesthetic curve, and showed how the family
of LAC with a given parameter $\alpha$ can be described as segments of a basic LAC after applying
the similarity transformations and shift of the arc length parameter (Proposition
\ref{prop:lac.blac}).  This constitutes the main result of this work.  In a second instance, we used
this characterization and provided an algorithmic way to recover the seven parameters
$p = (\alpha, S, s_0, l, \phi, x_0, y_0)$ that uniquely identify a given LAC segment. Finally, this
algorithm was later used to obtain the initial input for the gradient descent-like method used to
approximate a general planar curve segment by LAC (Equation \eqref{eq:opt}).
Regarding the applications, we developed this method aiming to be used as a reverse engineering tool
that characterizes existing objects. In view of this, we mention that further improvements in the
methodology employed to find the parameter $\alpha$ might be necessary, in particular, if we are
interested in characterize a given design by a concrete value of $\alpha$. Further analysis
regarding this will be considered in further works.

\section*{Acknowledgements} 
This work was supported by JSPS Kakenhi JP21K03329 and JST CREST Grant Number JPMJCR1911.


\begin{thebibliography}{}
\bibitem{Brander-Gravesen-Norbjerg:2017}
  D. Brander, J. Gravesen, T.B. N\o{}rjerg,
  Approximation by planer elastic curves,
  Adv. Comput. Math. {\bf 43}(2017) 25--43.
% DOI:10.1007/s10444-016-9474-z

\bibitem{doug73}
  D. Douglas and T. Peucker,
  Algorithms for the reduction of the number of points required to represent a digitized line or its caricature,
  The Canadian Cartographer {\bf 10} 2 (1973) 112-–122.

\bibitem{gobi11}
  R. U. Gobithaasan and K. T.  Miura,
  Aesthetic spiral for design,
  Sains Malaysiana {\bf 40} (2011) 1301--1305.

\bibitem{gobi14}
  R. U. Gobithaasan, Y. Siew Wei, K. T. Miura,
  Log-aesthetic curves for shape completion problem,
  J. Appl. Math. {\bf 2014} (2014) 960302.

\bibitem{GraiffZurita-Kajiwara-Suzuki:2022}
  S. E. Graiff Zurita, K. Kajiwara and T. Suzuki,
  Fairing of discrete planar curves by integrable discrete analogue of Euler’s elasticae,
  preprint, \verb\arXiv:2111.00804v1\ (2021).

\bibitem{hara94}
  T. Harada, N. Mori, K. Sugiyama,
  Study of quantitative analysis of curve's character,
  Bull. JSSD {\bf 40} (1994) 9--16. (in Japanese)

\bibitem{Hoffmann:MI_LN}
  T. Hoffmann, 
  Discrete differential geometry of curves and surfaces, 
  MI Lecture Notes {\bf 18}, Kyushu University, Fukuoka (2009).

\bibitem{inog18}
  J. Inoguchi, K. Kajiwara, K. T. Miura and M. Sato, W. K. Schief and Y. Shimizu,
  Log-aesthetic curves as  similarity geometric analogue of euler's elasticae,
  Comput. Aided Geom. Des. {\bf 61}(2018) 1--5. %DOI:10.1016/j.cagd.2018.02.002 

\bibitem{inog21}
  J. Inoguchi, Y. Jikumaru, K. Kajiwara, K. T. Miura and W. K. Schief,
  Log-aesthetic curves: similarity geometry, integrable discretization and variational principles,
  preprint, \verb\arXiv:1808.03104v2\ (2021)

\bibitem{Levien-Sequin:2009}
  R. Levien and C.H. S\'equin,
  Interpolating splines: which is the fairest of them all?
  Comput. Aided Des. Appl. {\bf 6}(2009) 91--102.
  % DOI:10.3722/cadaps.2009.91-102

\bibitem{miur06}
  K. T. Miura, J. Sone, A. Yamashita, T. Kaneko,
  Derivation of a general formula of aesthetic curves,
  Proceedings of the Eighth International Conference on Humans and Computers (HC2005) (2005) 166--171.

\bibitem{miur15}
  K. T. Miura and R. U. Gobithaasan,
  Aesthetic design with log-aesthetic curves and surfaces,
%  MI Lecture Notes {\bf 64}, Kyushu University, Fukuoka (2015) 103--112.
  in: Mathematical Progress in Expressive Image Synthesis III, eds. by Y. Dobashi and H. Ochiai, 
  Mathematics for Industry {\bf 24} (Springer, Singapore, 2015) 107--120. %DOI:10.1007/978-981-10-1076-7
  
\bibitem{rame72}
  U. Ramer,
  An iterative procedure for the polygonal approximation of plane curves,
  Comput. Graph. Image Process.  {\bf 1} (1972) 244--256.

\bibitem{ipopt}
  A. W\"achter and L.T. Biegler,
  On the implementation of an interior-point filter line-search algorithm for large-scale nonlinear programming,
  Math. Program. Ser. A {\bf 106} (2006) 25--57.

\bibitem{yosh06}
  N. Yoshida and T. Saito,
  Interactive aesthetic curve segments,
  Visual Comput {\bf 22} (2006) 896–-905.
  
\bibitem{ziat12}
  R. Ziatdinov, N. Yoshida, T. Kim,
  Analytic parametric equations of log-aesthetic curves in terms of incomplete gamma functions.
  Comput. Aided Geom. Des. {\bf 29} 2 (2012) 129--140.

\end{thebibliography}
\end{document}